\documentclass[a4paper,12pt]{amsart}
\pdfoutput=1 

\usepackage[top=2cm,bottom=2cm,outer=2.5cm,inner=2.5cm,marginpar=2.45cm]{geometry}
\usepackage[abbrev]{amsrefs}

\usepackage[french,ngerman,english]{babel}
\usepackage[utf8]{inputenc}
\usepackage[T1]{fontenc}

\usepackage[all,pdf]{xy}

\usepackage{enumitem}
\usepackage{mathtools}  
\usepackage[usenames,dvipsnames]{xcolor}
\usepackage{calc} 

\babeltags{de = ngerman}
\babeltags{fr = french}

\usepackage[no-mario]{marionotations}

\usepackage[Smaller]{cancel} 

\usepackage{bbm}     

\usepackage{circledsteps}  

\usepackage[final,colorlinks=true]{hyperref}
\usepackage{breakurl}  

\usepackage{subfiles}

\newcommand{\termKcolor}{NavyBlue}
\newcommand{\termJcolor}{VioletRed}

\newcommand{\alert}[2][RoyalBlue]{{\color{#1}#2}}

\newcommand{\defaultAmbientSpace}{X}

\newcommand{\defaultlcIndex}{\sigma}
\newcommand{\setDefaultlcIndex}[1]{\renewcommand{\defaultlcIndex}{#1}}

\newcommand{\defaultcohDegree}{q}

\newcommand{\defaultlclocus}{S}

\newcommand{\defaultvphi}{\vphi_L}

\newcommand{\defaultpsi}{\psi}

\newcommand{\defaultMetric}{\omega}

\NewDocumentCommand{\Ltwo}{ 
  D//{\bullet,\bullet}      
  D<>{defaultAmbientspace}  
  s                         
  m                         
}{L^{#1}_{(2)}\paren{\IfBooleanF{#3}{#2;} #4}}

\NewDocumentCommand{\Ltwosp}{
  D//{\defaultcohDegree}               
  t{M}                 
  D||{\defaultvphi}           
  o 
  D<>{\defaultMetric}  
}{\Ltwo/n,#1/*{D \otimes F \IfBooleanT{#2}{\otimes M}}_{#3, \IfNoValueF{#4}{#4,} #5}}

\NewDocumentCommand{\Harm}{ 
  D//{\defaultcohDegree}  
  D||{\defaultvphi}              
  o 
  D<>{\defaultMetric}     
}{\mathcal{H}^{n,#1}_{#2, \IfNoValueF{#3}{#3,} #4}}

\NewDocumentCommand{\lcIndex}{ 
  m  
  m  
  m  
}{#1\IfNoValueF{#2}{+#2}\IfNoValueF{#3}{-#3}}

\NewDocumentCommand{\lcData}{ 
  G{\defaultvphi}  
  O{\defaultpsi}   
  e{.}             
}{\paren{#1; \IfNoValueF{#3}{#3 \cdot} #2}}

\NewDocumentCommand{\lcdata}{ 
  s                
  d<>              
  G{\defaultvphi}  
  O{\defaultpsi}   
  e{.,}            
}{\newcommand{\datalist}{\IfNoValueF{#2}{#2,}#3,#4\IfNoValueF{#5}{,#5}\IfNoValueF{#6}{,#6}}
\IfBooleanTF{#1}{\datalist}{\paren{\datalist}}}

\newcommand{\spHbase}{\mathcal{H}}
\NewDocumentCommand{\spH}{ 
  D//{\defaultcohDegree}  
  t{M}                    
  s                       
  D||{\defaultlcIndex}    
  t{.}                    
  D||{#4 -1}              
  d()                     
}{\spHbase^{#1}\IfNoValueTF{#7}{
    \begingroup%
    \newcommand{\upidl}{\IfBooleanTF{#3}{
        \mtidlof{\vphi_{F \IfBooleanT{#2}{\otimes M}}}
      }{\aidlof|#4|{\vphi_{F \IfBooleanT{#2}{\otimes M}}}}
    }%
    \paren{\IfBooleanT{#2}{M\otimes}
      \IfBooleanTF{#5}{
        \frac{\upidl}{\aidlof|#6|{\vphi_{F \IfBooleanT{#2}{\otimes M}}}}
      }{\upidl}}
    \endgroup%
  }{\paren{\IfBooleanT{#2}{M\otimes}#7}}}

\DeclareMathOperator{\lc}{lc} 
\NewDocumentCommand{\lcc}{ 
  D||{\defaultlcIndex}       
  e{+-}                      
  D<>{\defaultAmbientSpace}  
  t{'}                       
  D(){\defaultlclocus}       
}{\lc_{#4}^{\lcIndex{#1}{#2}{#3}}\IfBooleanTF{#5}{\paren{#6}}{\lcData}}

\NewDocumentCommand{\lcS}{  
  s                       
  D(){\defaultlclocus}    
  D||{\defaultlcIndex}    
  e{+-}                   
  O{p}                    
}{\mathtt{\IfBooleanT{#1}{\rs} #2}^{\lcIndex{#3}{#4}{#5}}_{#6}}

\NewDocumentCommand{\PRes}{ 
  O{}      
  d()      
}{\mathcal R_{#1}\IfNoValueF{#2}{\paren{#2}}}

\newcommand{\defidlof}[1]{\mathcal{I}_{#1}}  
\NewDocumentCommand{\mtidlof}{   
  O{}      
  D<>{#1}  
  m        
}{\multidl_{#2}\paren{#3}} 

\NewDocumentCommand{\residlof}{  
  D||{\defaultlcIndex}   
  e{+-}                  
  d<>                    
  s                      
}{\sheaf R_{\IfNoValueTF{#4}{}{#4,} \lcIndex{#1}{#2}{#3}}\IfBooleanF{#5}{\lcData}}

\NewDocumentCommand{\Adjidlof}{
  D||{\defaultlcIndex}       
  D<>{\defaultAmbientSpace}  
  D(){\defaultlclocus}       
  m                          
}{\operatorname{\mathit{Adj}}^{#1}_{\paren{#2,#3}}\paren{#4}}

\NewDocumentCommand{\aidlof}{
  D||{\defaultlcIndex}   
  e{+-}                  
  d<>                    
  s                      
}{\sheaf{J}_{\!\IfNoValueTF{#4}{}{#4,} \lcIndex{#1}{#2}{#3}}\IfBooleanF{#5}{\lcData}}

\NewDocumentCommand{\lcV}{ 
  D||{\defaultlcIndex}    
  D//{\defaultvphi}       
  d()                     
  e{^}                    
  O{\defaultpsi}          
}{\:d\operatorname{lcv}^{#1\IfNoValueF{#4}{,\paren{#4}}}_{\IfNoValueF{#3}{#3,}#2}\left[#5\right]}

\NewDocumentCommand{\Ohvol}{ 
  D//{\defaultvphi} 
  d()               
  O{\defaultpsi}    
}{\dvol_{\IfNoValueF{#2}{#2,}#1}\left[#3\right]}

\newcommand{\dvol}{\:d\vol}

\newcommand{\RTFsym}{\mathfrak{F}} 
\NewDocumentCommand{\RTF}{ 
  s          
  G{\RTFsym} 
  d//        
  o          
  >{\SplitArgument{1}{,}} d<> 
  d||        
  d()        
  o          
}{%
  \begingroup%
    \newif\ifboolup%
    \booluptrue%
    \IfNoValueT{#4}{\IfNoValueT{#5}{\IfNoValueT{#6}{\boolupfalse}}}%
    \IfNoValueT{#7}{\boolupfalse}%
    \newcommand{\srptstr}{\cramped{{}^{\IfNoValueF{#4}{#4}\IfNoValueF{#5}{\inner#5}\IfNoValueF{#6}{\abs{#6}^2}}%
      \ifboolup _
      \fi{\ifboolup\displaystyle\fi\IfNoValueF{#7}{\paren{#7}}\IfNoValueF{#8}{%
          \ifboolup {\scriptstyle #8} \else _{#8} \fi%
        }}}}%
    \ifboolup%
      \IfBooleanTF{#1}{
        \smash[t]{
          \IfNoValueF{#3}{{}^{#3}}#2\raisebox{\depthof{$\srptstr$} * \real{0.3}}{$\srptstr$}%
        }%
      }{\IfNoValueF{#3}{{}^{#3}}#2\raisebox{\depthof{$\srptstr$} * \real{0.3}}{$\srptstr$}}%
    \else%
      \IfNoValueF{#3}{{}^{#3}}#2\srptstr%
    \fi%
  \endgroup%
}

\NewDocumentCommand{\mtlog}{O{e} d() D||{\defaultpsi}}{\log\!#1^{\paren{#2}}\abs{#3}}
\NewDocumentCommand{\slog}{O{e} D||{\defaultpsi}}{\log\abs{#1 #2}}
\NewDocumentCommand{\dlog}{O{e} D||{\defaultpsi}}{\mtlog[#1](2)|#2|}

\NewDocumentCommand{\logpole}{ 
  D||{\defaultpsi}       
  D//{\defaultlcIndex}   
  E{.^}{{e}{1+\eps}}     
  s                      
}{\abs{#1}^{#2} \IfBooleanTF{#5}{\slog[#3]|#1|}{\paren{\slog[#3]|#1|}^{#4}}}

\DeclareFontFamily{OMX}{MnSymbolE}{}
\DeclareSymbolFont{MnLargeSymbols}{OMX}{MnSymbolE}{m}{n}
\SetSymbolFont{MnLargeSymbols}{bold}{OMX}{MnSymbolE}{b}{n}
\DeclareFontShape{OMX}{MnSymbolE}{m}{n}{
    <-6>  MnSymbolE5
   <6-7>  MnSymbolE6
   <7-8>  MnSymbolE7
   <8-9>  MnSymbolE8
   <9-10> MnSymbolE9
  <10-12> MnSymbolE10
  <12->   MnSymbolE12
}{}
\DeclareFontShape{OMX}{MnSymbolE}{b}{n}{
    <-6>  MnSymbolE-Bold5
   <6-7>  MnSymbolE-Bold6
   <7-8>  MnSymbolE-Bold7
   <8-9>  MnSymbolE-Bold8
   <9-10> MnSymbolE-Bold9
  <10-12> MnSymbolE-Bold10
  <12->   MnSymbolE-Bold12
}{}
\DeclareMathDelimiter{\llangle}{\mathopen}%
{MnLargeSymbols}{'164}{MnLargeSymbols}{'164}
\DeclareMathDelimiter{\rrangle}{\mathclose}%
{MnLargeSymbols}{'171}{MnLargeSymbols}{'171}

\NewDocumentCommand{\idxup}{ 
  m                  
  O{\defaultMetric}  
  s                  
}{\paren{#1}^{\mathrlap{\!\IfBooleanTF{#3}{\smash[t]{#2}}{#2}}\makebox[\maxof{\widthof{$#2$}-\widthof{$\!\omega$}}{0pt}]{}}}

\NewDocumentCommand{\dep}{t{;} d<> O{\nu} m}{#4\IfBooleanTF{#1}{_}{^}{\IfNoValueF{#2}{#2\:}(#3)}}

\NewDocumentCommand{\sm}{s m}{{#2}\IfBooleanTF{#1}{_}{^}\text{sm}}

\NewDocumentCommand{\lelong}{m O{x}}{\operatorname{\boldsymbol{\nu}}\paren{#1,#2}}

\NewDocumentCommand{\rs}{ 
  s  
  m  
}{\IfBooleanTF{#1}{\smash[t]{\widetilde{#2}}}{\widetilde{#2}}}

\DeclareMathOperator{\Ann}{Ann}  
\DeclareMathOperator{\mlc}{mlc} 

\newcommand{\sect}[1][s]{\mathtt{#1}} 

\newcommand{\bphi}{\boldsymbol{\vphi}}

\NewDocumentCommand{\cbn}{  
  D//{\defaultlcIndex_V}
  D||{\defaultlcIndex}
}{\mathfrak{C}^{#1}_{#2}} 
\NewDocumentCommand{\Iset}{  
  D||{\defaultlcIndex}    
  e{+-}                   
  O{\defaultlclocus}      
  d()                     
}{I^{\lcIndex{#1}{#2}{#3}}_{#4}\IfNoValueF{#5}{\paren{#5}}} 

\ifcsname defineNoThmInMarionotations\endcsname
  \relax
\else 
  \newtheorem{prop}{Proposition}[subsection]
  
  \newtheorem{thm}[prop]{Theorem}
  \newtheorem{cor}[prop]{Corollary}

  \newtheorem{conjecture}[prop]{Conjecture}

  \theoremstyle{remark}
  \newtheorem{remark}[prop]{Remark}

  \theoremstyle{definition}
  \newtheorem{definition}[prop]{Definition}

  \numberwithin{equation}{subsection}

\fi

\allowdisplaybreaks  

\begin{document}

\newcommand{\titlestr}{%
  Residue functions and Extension problems%
}

\newcommand{\shorttitlestr}{%
  Residue functions and Extension problems%
}

\newcommand{\MCname}{Tsz On Mario Chan}
\newcommand{\MCnameshort}{Mario Chan}
\newcommand{\MCemail}{mariochan@pusan.ac.kr}

\newcommand{\addressstr}{%
  Dept.~of Mathematics, Pusan National
  University, Busan 46241, South Korea%
}

\newcommand{\subjclassstr}[1][,]{%
  32J25 (primary)#1  
  14B05 (secondary)
}

\newcommand{\keywordstr}[1][,]{%
  adjoint ideal sheaf#1
  multiplier ideal sheaf#1
  lc centre#1
  $L^2$ extension%
}

\newcommand{\dedicatorystr}{%
}

\newcommand{\thankstr}{%
  The author would like to thank Prof.~Kang-Tae Kim for his invitation
  to speak in the POSTECH Conference 2022 on Complex Analytic Geometry.
  This work was partly supported by the National Research Foundation (NRF) of
  Korea grant funded by the Korea government (No.~2018R1C1B3005963)%
}


\title[\shorttitlestr]{\titlestr}
 
\author[\MCnameshort]{\MCname}
\email{\MCemail}


\thanks{\thankstr}
 
\subjclass[2020]{\subjclassstr}



\begin{abstract}

The ``qualitative'' extension theorem of Demailly guarantees existence
of holomorphic extensions of holomorphic sections on some subvariety
under certain positive-curvature assumption, but that comes without any
estimate of the extensions, especially when the singular locus of the
subvariety is non-empty and the holomorphic section to be extended
does not vanish identically there.
Residue functions are analytic functions which connect the $L^2$ norms
on the subvarieties (or their singular loci) to $L^2$ norms with
specific weights on the ambient space.
Motivated by the conjectural ``dlt extension'', this note discusses
the possibility of retrieving the $L^2$ estimates for the extensions in
the general situation via the use of the residue functions.
It is also shown in this note that the $1$-lc-measure defined via the
residue function of index $1$ is indeed equal to the Ohsawa measure in
the Ohsawa--Takegoshi $L^2$ extension theorem.


\end{abstract} 

\date{\today} 

\maketitle



This note reviews the ``qualitative'' extension theorems obtained by
Demailly in \cite{Demailly_extension} and together with Junyan Cao and
Shin-ichi Matsumura in \cite{Cao&Demailly&Matsumura}, and then
discusses the possibility of retrieving the $L^2$ estimates for the
extensions.
This is the contents of Section \ref{sec:Demailly-extension}.
The residue functions introduced in \cite{Chan_adjoint-ideal-nas} by
the author are used to facilitate the re-establishment of the possible
estimates, which are discussed in Section \ref{sec:residue-functions}.

All results stated in this note have been proved somewhere else,
except for Proposition \ref{prop:1-lcV=Ohvol}, which states that the
$1$-lc-measure introduced in \cite{Chan&Choi_ext-with-lcv-codim-1} and
\cite{Chan_on-L2-ext-with-lc-measures} is indeed equal to the Ohsawa
measure in the Ohsawa--Takegoshi $L^2$ extension theorem, and
Corollary \ref{cor:1-aidl=Demailly-intermediate-multidl}, which
identifies the ad hoc ideal sheaf $\multidl'(m_{k-1})$ introduced by
Demailly in \cite{Demailly_extension}*{Def.~(2.11)} to be the adjoint
ideal sheaf of index $1$ introduced in \cite{Chan_adjoint-ideal-nas}.

\section{$L^2$ Extension theorem of Demailly}
\label{sec:Demailly-extension}

Based on the techniques developed through the Ohsawa--Takegoshi $L^2$
extension theorems (\cite{Ohsawa&Takegoshi-I}, \cite{Manivel},
\cite{Demailly_on_OTM-extension}, \cite{Siu_inv_plurigenera2}, ...),
Demailly proves in \cite{Demailly_extension} an extension theorem for
holomorphic sections on possibly non-reduced subvarieties defined by
some multiplier ideal sheaves on compact \textde{Kähler} manifolds.
More precisely, let
\begin{itemize}
\item $X$ be a compact \textde{Kähler} manifold, \footnote{The case
    where $X$ is a weakly pseudoconvex \textde{Kähler} manifold is
    also discussed in \cite{Demailly_extension}.}
  
\item $(L, e^{-\vphi_L})$ be a holomorphic line bundle over $X$
  equipped with a singular hermitian metric $e^{-\vphi_L}$,
  \footnote{The extension theorem of Demailly is also proved in
    \cite{Demailly_extension} for the case where $(L, e^{-\vphi_L})$
  is replaced by $(E,h)$, a holomorphic vector bundle $E$ over $X$
  equipped with a \emph{smooth} hermitian metric $h$.}
  
\item $\psi \leq -1$ be a bounded global function on $X$
   with neat analytic singularities.
\end{itemize}
Here a function $\vphi$ is said to have \emph{neat analytic
  singularities} if it is locally the difference $\vphi_1 - \vphi_2$
of quasi-plurisubharmonic (quasi-psh) functions of the form
\begin{equation*}
  \vphi_i = c_i \log\paren{\sum_{j=1}^N \abs{g_{ij}}^2} + \alpha_i \; ,
\end{equation*}
where $c_i \in \fieldR_{\geq 0}$, $g_{ij} \in \holo_X$ and $\alpha_i
\in \smooth_X$ for $i=1,2$.
The function $\vphi$ is said to have \emph{analytic singularities} if
the local functions $\alpha_i$ are bounded (need not be smooth).

Suppose that $\vphi_L$ and $\vphi_L +m\psi$ are quasi-psh for some $m
\geq 0$ so that
the multiplier ideal sheaves $\mtidlof{\vphi_L}$ and $\mtidlof{\vphi_L+m\psi}$
are coherent.
The subvarieties considered in \cite{Demailly_extension} (also in
\cite{Cao&Demailly&Matsumura}) are those defined by the annihilator of
$\frac{\mtidlof{\vphi_L}}{\mtidlof{\vphi_L +m \psi}}$, denoted by
$Y^{(m)}$.
(When $\vphi_L = 0$ and $\psi = \log\abs{z_1}^2$ on a coordinate
neighbourhood,
$\Ann_{\holo_X}\paren{\frac{\mtidlof{\vphi_L}}{\mtidlof{\vphi_L +m
      \psi}}} = \mtidlof{m \psi} =\defidlof{\set{z_1 = 0}}^{\floor
  m}$, where $\defidlof{\set{z_1 = 0}}$ is the defining ideal sheaf of
$\set{z_1 = 0}$ and $\floor m$ is the round-down of $m$.
The subvariety $Y^{(m)}$ is then non-reduced for $m \geq 2$.)

\subsection{Extension theorem without estimates}

The extension result is stated as follows.
\begin{thm}[\cite{Demailly_extension}*{Thm.~(2.14b)} and
  \cite{Cao&Demailly&Matsumura}*{Thm.~1.1}] \label{thm:Demailly-extension}
  Given any fixed real number $m > 0$,
  suppose that there is a constant $\delta > 0$ such that
  \begin{equation*}
    \ibddbar\paren{\vphi_L +\beta \psi} \geq 0
    \quad\text{ for all } \beta \in [m, m +\delta]
  \end{equation*}
  in the sense of currents.
  Then the restriction morphism
  \begin{equation*}
    \cohgp 0[X]{K_X \otimes L \otimes \mtidlof{\vphi_L}}
    \to \cohgp 0[Y^{(m)}]{
      K_X \otimes L \otimes
      \frac{\mtidlof{\vphi_L}}{\mtidlof{\vphi_L +m\psi}}
    }
  \end{equation*}
  is surjective.
\end{thm}
\begin{remark}
  \cite{Cao&Demailly&Matsumura}*{Thm.~1.1} indeed states that the
  corresponding statement on holomorphically convex \textde{Kähler}
  manifold and for higher cohomology groups also holds true.
\end{remark}

This theorem, while being proved via $L^2$ method, does not require
(explicitly) the convergence of any $L^2$ norms of the sections to be
extended from the subvariety.
This is considered as an advantage since the \emph{Ohsawa measure} on the
subvariety $Y^{(m)}$, the measure appears in the estimate of the
Ohsawa--Takegoshi $L^2$ extension theorem (see \cite{Ohsawa-V} or
\cite{Demailly_extension}*{(2.4)}), \mmark{diverges in general around the
  singular locus of $Y^{(m)}$ (see \cite{KimDano&Seo_adj-idl} or
  \cite{KimDano-L2-ext-for-lc} for more discussion on the
  singularities on the Ohsawa measure)}{Reference to Dano's paper? \alert{Done.}}.
The classical theorem, which requires the convergence of sections with
respect to the Ohsawa measure, is deemed inapplicable to extend
sections non-vanishing on the singular locus.
It was hoped that this feature of Theorem \ref{thm:Demailly-extension}
can be exploited in order to solve the conjectural ``dlt extension'' in
\cite{DHP}*{Conj.~1.3}.
The hurdle is, while the sections which can be extended via Theorem
\ref{thm:Demailly-extension} have to sit inside a quotient of
multiplier ideal sheaves (namely,
$\frac{\mtidlof{\vphi_L}}{\mtidlof{\vphi_L+m\psi}}$), the conjecture of
the ``dlt extension'' demands the extension of sections which are not
confined in any multiplier ideal sheaves. 
One either has to show that the quotient of the multiplier ideal sheaves of
some suitably chosen potentials in the setup of the ``dlt extension''
is trivial, or has to improve Theorem \ref{thm:Demailly-extension} in
the specific case so that any holomorphic sections on the
corresponding $Y^{(m)}$ can also be extended.
At the time of writing, the author does not know any successful
attempt in the latter approach.
For the former approach, the known strategy involves the use of the
$L^2$ estimates with universal constant from the Ohsawa--Takegoshi
$L^2$ extension theorem (cf.~proof of the ``plt extension'' in
\cite{DHP}*{Thm.~1.7} or
\cite{Chan&Choi_ext-with-lcv-codim-1}*{Thm.~1.6.1}, which, roughly
speaking, corresponds to the case where the subvariety $Y^{(m)}$ is
smooth).
The ``universal constant'' here means that the multiplicative constant
involved in the estimate is independent of the involving sections and
metrics.

However, as a trade-off of the non-requirement of the convergence of
any $L^2$ norm on $Y^{(m)}$, Theorem \ref{thm:Demailly-extension}
does not provide any estimate for the extensions in general.
In \cite{Demailly_extension}, it is shown that an $L^2$ estimate can
be obtained for extensions corresponding to successive jumping
numbers, which is explained below.

\subsection{Extension theorem with $L^2$ estimates}
\label{sec:ext-thm-with-estimates}

Thanks to the solution to the strong openness conjecture for psh
functions by Guan and Zhou (\cite{Guan&Zhou_openness}, see also
\cite{Guan&Zhou_effective_openness} and \cite{Hiep_openness}), under
the compactness assumption on $X$, there exists a strictly increasing
sequence of \emph{jumping numbers}
\begin{equation*}
  0 \leq m_0 < m_1 < \dotsm < m_k < \dotsm
\end{equation*}
such that, for each $k \in \Nnum$,
\begin{equation*}
  \mtidlof{\vphi_L +m_{k} \psi}
  \subsetneq
  \mtidlof{\vphi_L +m \psi}
  =\mtidlof{\vphi_L +m_{k-1} \psi}
  \quad\text{ on $X \;$ for all } m \in [m_{k-1}, m_k) \; .
\end{equation*}

\begin{remark}
  If $\vphi_L$ (as well as $\psi$) has only analytic singularities,
  then the sequence of jumping numbers $\seq{m_k}_{k \in \Nnum}$ has
  no accumulation point, as can be seen via a log-resolution of the
  polar ideal sheaves of $\vphi_L$ and $\psi$ (whose existence is
  guaranteed by Hironaka's result \cite{Hironaka}).
  The number $m_0$ can be set to $0$ and the sequence diverges to
  $+\infty$ in this case.
  When $\vphi_L$ has more general singularities, the sequence
  $\seq{m_k}_{k \in \Nnum}$ may converge to $\lim_{k \tendsto +\infty}
  m_k =: \dep[1]m_0 < +\infty$ (see
  \cite{Guan&Li_adjIdl-not-coherent} for an example of such $\vphi_L$
  (with $\psi :=\log\abs{z_1}^2$ and $\dep[1]m_0 = 1$) and
  \cite{KimDano&Seo_jumping-numbers}*{\S 5} for some further
  discussion).
  The strong openness property again guarantees the existence of
  another sequence of jumping numbers
  \begin{equation*}
    0 < \dep[1]m_0 < \dep[1]m_1 < \dotsm < \dep[1]m_k < \dotsm
  \end{equation*}
  satisfying the same property of $\seq{m_k}_{k\in\Nnum}$ on the
  family $\seq{\mtidlof{\vphi_L +m\psi}}_{m \in \fieldR_{\geq 0}}$ of
  multiplier ideal sheaves.
  Since only extensions corresponding to successive jumping numbers
  ($m_{k-1}$ and $m_k$ or $\dep[1]m_{k-1}$ and $\dep[1]m_k$) is under
  concern in this section, for the sake of generality, the number
  $m_0$ is not assumed to be $0$ as does in \cite{Demailly_extension}
  and it is assumed that $m_k$ is not an accumulation point of jumping
  numbers in what follows.
\end{remark}

The subvariety $S :=S^{(m_k)}$ ($\subset Y^{(m_k)}$) defined by
$\Ann_{\holo_X} \paren{\frac{ 
    \mtidlof{\vphi_L +m_{k-1} \psi}
  }{
    \mtidlof{\vphi_L +m_{k} \psi}
  }}$, which is the scheme-theoretic difference between
$Y^{(m_k)}$ and $Y^{(m_{k-1})}$, is \emph{reduced} (see
\cite{Demailly_extension}*{Lemma (4.2)}).
The statement of a ``quantitative'' extension in
\cite{Demailly_extension} is recalled as follows.

\begin{thm}[\cite{Demailly_extension}*{Thm.~(2.12a)}] \label{thm:Demailly-extension-with-estimate}
  Given a fixed jumping number $m_k$ (which is not an accumulation
  point) of the family $\set{\mtidlof{\vphi_L+m\psi}}_{m \in
    \fieldR_{\geq 0}}$, suppose that there is a constant $\delta > 0$
  such that one has the curvature assumption
  \begin{equation*}
    \ibddbar\paren{\vphi_L +\beta \psi} \geq 0
    \quad\text{ for all } \beta \in [m_k, m_k +\delta]
  \end{equation*}
  in the sense of currents.
  If $f \in \cohgp 0[S]{K_X \otimes L \otimes \frac{
      \mtidlof{\vphi_L +m_{k-1}\psi}
    }{
      \mtidlof{\vphi_L +m_{k}\psi}
    }}$ has finite $L^2$ norm with respect to the (generalised) Ohsawa
  measure, i.e.
  \begin{equation*}
    \int_S \abs{J^{m_k} f}^2 \Ohvol
    :=\lim_{t \tendsto -\infty} \;\;\smashoperator{\int\limits_{\set{t <
          \psi < t+1} \mathrlap{\subset X}}} \;\;
    \abs{\rs*f}^2 \:e^{-\vphi_L -m_k\psi} < +\infty \; ,
  \end{equation*}
  where $\rs f \in K_X \otimes L \otimes \smooth \cdot
  \mtidlof{\vphi_L +m_{k-1}\psi} (S)$ is some smooth extension of
  $f$, then there exists $F \in \cohgp 0[X]{K_X \otimes L \otimes
    \mtidlof{\vphi_L +m_{k-1}\psi}}$ which is a holomorphic extension
  of $f$, i.e.~$F \equiv f \mod \mtidlof{\vphi_L+m_k\psi}$ on $X$,
  such that
  \begin{equation*}
    \int_X \frac{\abs F^2 \:e^{-\vphi_L -m_k \psi}
    }{\abs{\delta \psi}^2 +1}
    \leq \frac{34}{\delta}
    \int_S \abs{J^{m_k} f}^2 \Ohvol \; .
  \end{equation*}
\end{thm}

\begin{remark} \label{rem:extension-thm-with-1-lc-measure}
  In \cite{Chan&Choi_ext-with-lcv-codim-1}*{Thm.~1.4.5 and
    Thm.~3.4.1}, the estimate, at least in the case where $\vphi_L$
  has only neat analytic singularities, is improved to
  \begin{equation*}
    \int_X \frac{\abs F^2 \:e^{-\vphi_L -m_k \psi}
    }{\logpole//.\ell^2}
    \leq
    \int_S \abs{f}^2 \lcV|1|^{m_k} 
  \end{equation*}
  (see \cite{Chan_on-L2-ext-with-lc-measures}*{Remark 1.1.4} for an
  explanation for the slightly different form from
  \cite{Chan&Choi_ext-with-lcv-codim-1}*{Thm.~1.4.5} on the
  left-hand-side), where $\ell \geq e$ is a sufficiently large
  constant, depending only on $\delta$ and $\psi$, and the
  right-hand-side is the $L^2$ norm of $f$ with respect to the
  $1$-lc-measure introduced in \cite{Chan&Choi_ext-with-lcv-codim-1},
  which is also discussed in Section \ref{sec:residue-functions}.
  The constant $\ell$ is the multiplicative constant in the classical
  Ohsawa--Takegoshi extension theorem in disguise and is ``universal''
  in this case (it does not depend on $\vphi_L$, $m_k$ and the
  involving sections $F$ and $f$).
  It is shown in Proposition \ref{prop:1-lcV=Ohvol} below that, indeed,
  the $1$-lc-measure is equal to the Ohsawa measure.
\end{remark}

Let $\multidl'(m_{k-1})$ be the subsheaf of $\mtidlof{\vphi_L+m_{k-1}\psi}$
consisting of all the germs $F \in \mtidlof{\vphi_L+m_{k-1}\psi}_x$
which are locally $L^2$ with respect to the Ohsawa measure, as defined
in \cite{Demailly_extension}*{Def.~(2.11)}.
It follows that $\mtidlof{\vphi_L +m_k\psi} \subset \multidl'(m_{k-1})
\subset \mtidlof{\vphi_L+m_{k-1}\psi}$, and Theorem
\ref{thm:Demailly-extension-with-estimate} indeed states that, under
the compactness assumption on $X$ (so that locally $L^2$ implies
globally $L^2$), all sections in $\cohgp 0[S]{K_X\otimes L\otimes
  \frac{
    \multidl'(m_{k-1})
  }{\mtidlof{\vphi_L +m_k\psi}}}$ can be extended holomorphically to
some $F \in \cohgp 0[X]{K_X\otimes L\otimes \multidl'(m_{k-1})}$ with
$L^2$ estimates.
Given the fact that Theorem \ref{thm:Demailly-extension} also guarantees
the existence of extensions of sections taking values in the
complement $\frac{
  \mtidlof{\vphi_L +m_{k-1}\psi}
}{
  \mtidlof{\vphi_L +m_{k}\psi}
} \setminus \frac{
  \multidl'(m_{k-1})
}{
  \mtidlof{\vphi_L +m_{k}\psi}
}$, the goal of this study is to obtain reasonable estimates for these
extensions.
The residue functions and the corresponding adjoint ideal sheaves
discussed in Section \ref{sec:residue-functions} are introduced to
facilitate the quest.
The sheaf $\multidl'(m_{k-1})$ indeed equals
$\aidlof|1|{\vphi_L}.{m_k}$, the adjoint ideal sheaf of index $1$
defined via the residue functions (see Corollary
\ref{cor:1-aidl=Demailly-intermediate-multidl}).

If the desired estimate for some holomorphic extension of every
section taking values in $\frac{
  \mtidlof{\vphi_L +m_{k-1}\psi}
}{
  \mtidlof{\vphi_L +m_{k}\psi}
}$ can be obtained for every $k \geq 1$, then, in view of the short
exact sequence
\begin{equation*}
  \renewcommand{\objectstyle}{\displaystyle}
  \xymatrix@R=0.1cm{
    {0} \ar[r] &
    {\frac{
        \mtidlof{\vphi_L +m_{\ell}\psi}
      }{
        \mtidlof{\vphi_L +m_{k}\psi}
      }} \ar[r] &
    {\frac{
        \mtidlof{\vphi_L +m_{\ell-1}\psi}
      }{
        \mtidlof{\vphi_L +m_{k}\psi}
      }} \ar[r] &
    {\frac{
        \mtidlof{\vphi_L +m_{\ell-1}\psi}
      }{
        \mtidlof{\vphi_L +m_{\ell}\psi}
      }} \ar[r] &
    {0}
  }
\end{equation*}
for all integers $k \geq \ell \geq 1$, one can also obtain an
extension with estimate for any $f$ taking values in $\frac{
  \mtidlof{\vphi_L +m_{0}\psi}
}{
  \mtidlof{\vphi_L +m_{k}\psi}
}$ by treating $f$ as an $\frac{
  \mtidlof{\vphi_L +m_{0}\psi}
}{
  \mtidlof{\vphi_L +\alert{m_{1}}\psi}
}$-valued section to obtain an $\mtidlof{\vphi_L+m_0\psi}$-valued
extension $F_0$ on $X$ with estimate and repeating the process to $f -
F_0 \bmod \mtidlof{\vphi_L +m_k\psi}$, which is now $\frac{
  \mtidlof{\vphi_L +\alert{m_{1}}\psi}
}{
  \mtidlof{\vphi_L +m_{k}\psi}
}$-valued.
Iterating this procedure results in $\mtidlof{\vphi_L+m_j\psi}$-valued
sections $F_j$ for $j =0, \dots, k-1$, each having an estimate, and
the sum $\sum_{j = 0}^{k-1} F_j$ is an extension of $f$.

\begin{remark}
  At the moment of writing, the author cannot even make a prediction
  on whether it is possible to obtain a holomorphic extension, with
  estimate, for an $\frac{
    \mtidlof{\vphi_L +m_{0}\psi}
  }{
    \mtidlof{\vphi_L +\dep[1]m_{0}\psi}
  }$-valued section, where $\dep[1]m_{0} := \lim_{k \tendsto +\infty}
  m_k < +\infty$. 
\end{remark}

\section{Residue functions and adjoint ideal sheaves}
\label{sec:residue-functions}

The definition of residue functions is based on the following
model:
given the function $\psi :=\sum_{j=1}^\sigma \log x_j -1$ on the cube
$[0,1]^n$ (where $\sigma \leq n$) and any compactly supported smooth
function $G \in \smooth_c\paren{[0,1)^n}$, a direct computation with
Fubini's theorem and integration by parts yields
\begin{align*}
  \RTF[G](\eps)[s]
  &:=\eps \int_{[0,1]^n} \frac{G \:dx_1 \dotsm dx_n}{x_1 \dotsm x_\sigma
    \:\logpole/s/} \\
  &\: =
    \begin{aligned}[t]
      &-\eps c_{1}(s,\sigma,\eps) \:\RTF[G](1+\eps)[s] -\dotsm
      -\eps c_{\sigma -1}(s,\sigma,\eps) \:\RTF[G](\sigma -1 +\eps)[s]
      \\
      &~+\frac{\paren{-1}^\sigma}{(\sigma-1)!}
      \int_{[0,1]^n}
      \frac{\diff^\sigma G}{\diff {x_\sigma} \dotsm \diff {x_1}} \:
      \frac{\:dx_1\dotsm dx_n}{\paren{\log\abs{e\psi}}^\eps} 
      & \text{if } s=\sigma \; ,
    \end{aligned}
  \\
  &\hphantom{=}~\mathllap{\xrightarrow{\eps \tendsto 0^+}}~
    \frac{1}{(\sigma-1)!} \int_{[0,1]^{n-\sigma}}
    \res G_{\set{x_1 =\dotsm =x_\sigma =0}} \:dx_{\sigma+1} \dotsm dx_n \; ,
\end{align*}
where $\eps > 0$ and $c_j(s,\sigma,\eps) > 0$ for $j=1,\dots,\sigma-1$
are positive coefficients which are polynomials in $\eps$ (see the
proof of \cite{Chan_on-L2-ext-with-lc-measures}*{Prop.~2.2.1} for the
computation in a more general setup).
One can also show that the integral $\RTF[G](\eps)[s]$ diverges for
any $\eps > 0$ when $s < \sigma$ and $\res G_{\set{x_1 =\dotsm
    =x_\sigma =0}} \not\equiv 0$, and $\lim_{\eps \tendsto 0^+}
\RTF[G](\eps)[s] = 0$ when $s > \sigma$.
Using the formula obtained after applying successive integration by
parts, the function $\eps \mapsto \RTF[G](\eps)[s]$ can be continued
analytically to the whole complex plane (see
\cite{Chan_on-L2-ext-with-lc-measures}*{Thm.~2.3.1}).
This illustrates that such kind of functions connect analytically the
norm on an subvariety (at $\eps =0$) with a norm on the ambient space
with a specific weight (at some $\eps > 0$).

\subsection{Definitions of residue functions and related notions}

The definition of residue functions in the setup as in Section
\ref{sec:Demailly-extension} is given as follows.

\begin{definition}[\cite{Chan_on-L2-ext-with-lc-measures}*{Def.~1.1.1}]
  Given the potential $\vphi_L$, the function $\psi \leq -1$ and the
  jumping numbers $\set{m_k}_{k \in \Nnum}$ described as in Section
  \ref{sec:Demailly-extension}, on any open set $V \subset X$, the
  \emph{residue function $\fieldR_{>0} \ni \eps \mapsto
    \RTF[G](\eps)[V,\sigma]$ of index $\sigma$ for any $L \otimes \conj
    L$-valued $(n,n)$-form $G$ with respect to the data
    $(V,\vphi_L,\psi,m_k)$} is given by
  \begin{equation*}
    \RTF[G](\eps)[V,\sigma]
    := \RTF[G](\eps)[V,\vphi_L,\psi,m_k,\sigma]
    := \eps \int_V \frac{
      G \:e^{-\vphi_L -m_k\psi}
    }{
      \logpole
    } \quad\text{ for } \eps > 0 \; .
  \end{equation*}
\end{definition}

When $\vphi_L$ has only neat analytic singularities and $G =\abs f^2$
for some $f \in K_X \otimes L \otimes \mtidlof{\vphi_L +m_{k-1}\psi}
(V)$ such that $\RTF|f|(\eps)[V,\sigma] < +\infty$ for all $\eps > 0$,
the function can then be continued analytically to an entire function
(see \cite{Chan_on-L2-ext-with-lc-measures}*{Thm.~2.3.1} with the
log-resolution of $(X,\vphi_L,\psi)$ described as in
\cite{Chan_adjoint-ideal-nas}*{\S 2.3} considered).
The value $\RTF|f|(0)[V,\sigma]$, called the \emph{residue norm of $f$ on
the $\sigma$-lc centres of $(X,\vphi_L,\psi,m_k)$}, is indeed the
$L^2$ norm with respect to the $\sigma$-lc-measure introduced in
\cite{Chan&Choi_ext-with-lcv-codim-1}*{Def.~1.4.3}.
In order to describe the supports of these $\sigma$-lc-measures more
properly, the following version of adjoint ideal sheaves is
introduced.

\begin{definition}[\cite{Chan_adjoint-ideal-nas}*{Def.~1.2.1}]
  The \emph{(analytic) adjoint ideal sheaf of index
    $\sigma$ with respect to the data $\lcdata<X>.{m_k}$}, denoted
  by $\aidlof :=\aidlof{\vphi_L}.{m_k}$, is an ideal sheaf on $X$ such
  that its stalk at each $x \in X$ is given by
  \begin{equation*}
    \aidlof{\vphi_L}.{m_k}_x
    =\setd{f\in \holo_{X,x}}{\exists~\text{open set } V_x \ni x \: , \;
      \forall~\eps > 0 \: , \;
      \RTF|f|(\eps)[V_x,\vphi_L,\psi,m_k,\sigma] < +\infty } \; . 
  \end{equation*}
\end{definition}

When $\vphi_L$ has only neat analytic singularities, it is shown in
\cite{Chan_adjoint-ideal-nas}*{Thm.~4.1.2} that there exists an
integer $\sigma_{\mlc} \in [0, n]$ such that
\begin{equation*}
  \xymatrix@C=0.2cm@R=0.4cm{
    {\aidlof|0|.{m_k}} \ar@{}[r]|-*{\subset} &
    {\aidlof|1|.{m_k}} \ar@{}[r]|-*{\subset} &
    {\dotsm} \ar@{}[r]|-*{\subset} &
    {\aidlof|\sigma_{\mlc}|.{m_k}} \ar@{}[r]|-*{=} &
    {\aidlof|\sigma_{\mlc}+1|.{m_k}} \ar@{}[r]|-*{=} &
    {\dotsm} \\
    {\mtidlof{\vphi_L+m_k\psi}} \ar@{}[u]|(.4)*[left]{=}
    &&&
    {\mtidlof{\vphi_L+m_{k-1}\psi}} \ar@{}[u]|(.4)*[left]{=}
  }
\end{equation*}
The filtration gives rise to the lc centres. 

\begin{definition}[\cite{Chan_adjoint-ideal-nas}*{Def.~1.2.4}] \label{def:sigma-lc-centres}
  A \emph{$\sigma$-lc centre of $(X,\vphi_L,\psi,m_k)$} is an irreducible
  component $\lcS$ of the reduced subvariety $\lcc :=\lcc.{m_k} =\bigcup_{p
    \in \Iset} \lcS$ in $X$ defined by
  the annihilator
  \begin{equation*}
    \defidlof{\lcc.{m_k}}
    :=\Ann_{\holo_X}\paren{\frac{
      \aidlof{\vphi_L}.{m_k}
    }{
      \aidlof-1{\vphi_L}.{m_k}
    }}
  \end{equation*}
  (see \cite{Chan_adjoint-ideal-nas}*{Thm.~5.2.1} for a proof on
  $\lcc.{m_k}$ being reduced).
\end{definition}
Note that one has $S = \lcc|1| \cup \lcc|2| \cup \dotsm \cup
\lcc|\sigma_{\mlc}|$ by \cite{Chan_adjoint-ideal-nas}*{Prop.~5.2.6}. 
When $\vphi_L$ has only neat analytic singularities, 
after passing to a log-resolution of $(X,\vphi_L,\psi)$ as discussed
in \cite{Chan_adjoint-ideal-nas}*{\S 2.3} such that $S$ is a reduced
snc divisor in particular, each $\sigma$-lc centre $\lcS$ defined in
Definition \ref{def:sigma-lc-centres} is just a component of the
intersection of any $\sigma$ distinct pieces of irreducible components
of $S$, which coincides with the lc centre of codimension $\sigma$ of
the log-smooth and lc pair $(X,S)$ in the study of minimal model
program (see \cite{Kollar_Sing-of-MMP}*{Def.~4.15}; see also
\cite{Chan_adjoint-ideal-nas}*{Example 6.2.1} for an example on which
the two concepts of lc centres differ).
Moreover, in this case, one has $\aidlof.{m_k}
=\mtidlof{\vphi_L+m_{k-1}\psi} \cdot \defidlof{\lcc+1.{m_k}}$ (see
\cite{Chan_adjoint-ideal-nas}*{Thm.~4.1.2}).

A direct computation (see
\cite{Chan&Choi_ext-with-lcv-codim-1}*{Prop.~2.2.1}) shows that, for
any $f \in \aidlof(\cl V)$ on the closure of any open set $V$ in $X$,
the value $\RTF|f|(0)[V,\sigma]$ is finite and is an $L^2$ norm of $f$
on $\lcc$.
In view of this, for any $f \in \cohgp 0[S]{K_X \otimes L \otimes \frac{
    \mtidlof{\vphi_L+m_{k-1}\psi}
  }{
    \mtidlof{\vphi_L+m_{k}\psi}
  }}$, the
\emph{$\sigma$-lc-measure with respect to $(\vphi_L,\psi,m_k, f)$} is
defined to be the measure $\abs f^2 \lcV^{m_k}$ given by the functional
\begin{equation*}
  \xymatrix@R=0.1cm{
    {\smooth[0]_c\paren{\lcc.{m_k} \setminus \lcc+1.{m_k}}} \ar[r]
    \ar@{}[d]|*[left]{\in} &
    {\fieldR} \ar@{}[d]|*[left]{\in}
    \\
    *+/r 1.8cm/{g} \ar@{|->}[r]+/l 3.8cm/ &
    {\displaystyle \RTF[g]|f|(0)[X,\sigma]
    =: \int_{\mathrlap{\lcc.{m_k}}} \qquad g \:\abs f^2 \lcV^{m_k} \; .}
  }
\end{equation*}
It follows from \cite{Chan_adjoint-ideal-nas}*{Thm.~4.1.2} that, if
the function $f$ takes values in $\frac{\aidlof|\alert{s}|}{\aidlof|\alert{s-1}|}$,
the measure $\abs f^2 \lcV|\alert{\sigma}|^{m_k}$ is non-trivial and is nowhere
divergent if and only if $s = \sigma$. 

Another use of such adjoint ideal sheaves can be found in
\cite{Chan&Choi_injectivity-I}.

\subsection{Relation with the Ohsawa measure and extension theorem of
  Demailly}

The following proposition shows that the $1$-lc-measure is nothing
other than the Ohsawa measure.
\begin{prop} \label{prop:1-lcV=Ohvol}
  Assume that $\vphi_L$ has only neat analytic singularities.
  For any $f \in \cohgp 0[S]{K_X \otimes L \otimes
    \frac{
      \mtidlof{\vphi_L +m_{k-1}\psi}
    }{
      \mtidlof{\vphi_L +m_{k}\psi}
    }}$ and $g \in \smooth[0]_c\paren{S \setminus \Sing(S)}$,
  \begin{equation*}
    \RTF[g]|f|(0)[X,1] =\int_{S} g \:\abs{f}^2 \lcV|1|^{m_k}
    =\int_S g \:\abs{J^{m_k} f}^2 \Ohvol \; .
  \end{equation*}
  In other words, $\abs{f}^2 \lcV|1|^{m_k} =\abs{J^{m_k} f}^2 \Ohvol$
  in the sense of measures (which can possibly diverge around
  $\Sing(S)$).
\end{prop}

\begin{proof}
  \setDefaultlcIndex{1}
  
  Passing to a log-resolution of $(X,\vphi_L,\psi)$ (see
  \cite{Chan_adjoint-ideal-nas}*{\S 2.3}), one can assume
  that $\psi^{-1}(-\infty)$ (hence $S$), $\vphi_L^{-1}(-\infty)$ and
  $\psi^{-1}(-\infty) \cup \vphi_L^{-1}(-\infty)$ are all snc divisors
  (the triple $\lcdata<X>$ is said to be in an \emph{snc configuration}).
  Since $g$ has compact support away from $\Sing(S)$, by decomposing
  $g$ into a sum of functions supported on different components of
  $S$, one can assume that $g$ is supported only on a single
  component, say, $D_1$ $(:= \lcS|1|[1])$, of $S$ (while remaining in
  $\smooth[0]_c\paren{S\setminus \Sing(S)}$).
  Via the use of a partition of unity, one can assume further that $g$
  is supported on $V \cap D_1$, where $V$ is an admissible open set in
  $X$ such that $V \cap S =\set{z_1 \dotsm z_{\sigma_V} = 0}$ and $V
  \cap D_1 =\set{z_1 = 0}$ (see \cite{Chan_adjoint-ideal-nas}*{\S 4.1}
  for more precise definition of admissible open sets in this
  context).
  Moreover, $\res\psi_V$ can be expressed as
  \begin{equation*}
    \res\psi_V = \sum_{j=1}^{\sigma_V} \nu_j \log\abs{z_j}^2
    +\smashoperator[r]{\sum_{k=\sigma_V+1}^n} a_k \log\abs{z_k}^2 +\alpha \; ,
  \end{equation*}
  where $\nu_j > 0$ for $j=1,\dots, \sigma_V$ and $a_k \geq 0$ for
  $k=\sigma_V+1, \dots, n$ are constants and $\alpha$ is a
  smooth function on $V$.
  The set $V$ can be decomposed into $U \times W$ where $U$ is a $1$-disc
  about the origin in the $z_1$-plane and $W \isom V \cap D_1$.

  Following the discussion in \cite{Chan_adjoint-ideal-nas}*{\S 2.3}
  (in particular, \cite{Chan_adjoint-ideal-nas}*{Remark 2.3.8 and
    Lemma 2.3.9}), there exist an effective snc ($\Znum$-)divisor
  $S_0$ with $\supp S_0 \subset S$ and a quasi-psh potential $\bphi$
  on $L \otimes S_0^{-1} \otimes S^{-1}$ such that
  \begin{equation} \label{eq:potential-decomposition}
    \vphi_L +m_k\psi = \bphi +\phi_{S_0} +\phi_{S} \; ,
  \end{equation}
  where $\phi_{S_0}$ and $\phi_S$ are potentials defined from some
  holomorphic canonical sections of $S_0$ and $S$ respectively,
  and $e^{-\bphi}$ is locally integrable at general points of $S$ in
  $X$.
  Indeed, since $\bphi$ has only neat analytic singularities with snc,
  this means that $\bphi^{-1}(-\infty) \not\supset D_j$ for any
  irreducible component $D_j$ of $S$.
  Moreover, given the canonical section $\sect_0$ of $S_0$ such that
  $\phi_{S_0} =\log \abs{\sect_0}^2$ and any local lifting $\rs f \in
  K_X \otimes L \otimes \mtidlof{\vphi_L +m_{k-1}\psi}\paren{V}$
  of $f$, one has $\rs f_0 :=\frac{\rs f}{\sect_0}$ being holomorphic
  (see \cite{Chan_adjoint-ideal-nas}*{Remark 2.3.8}).
  Also write $\rs f_0 =: \rs f_0' \:dz_1 \wedge \dotsm \wedge dz_n$
  and $\abs{\rs* f_0}^2 := \abs{\rs* f_0'}^2 \bigwedge_{j=1}^n
  \pi\ibar \:dz_j \wedge d\conj{z_{j}} =\abs{\rs* f_0'}^2 \dvol_V$
  (where $\ibar := \ibardefn \;$ \ibarfootnote).

  Write $\lcS[(1)] := S -D_1$ for convenience.
  Consider the projection $\pr \colon V \isom U \times W \to W \isom
  V\cap D_1$ given by $(z_1, z_2, \dots, z_n) \mapsto (z_2, \dots,
  z_n)$.
  Choose a compactly supported smooth cut-off function $\rho \colon U
  \to [0,1]$ such that $\rho \equiv 1$ on a neighbourhood of the
  origin.
  Let $\rs g := \rho(z_1) \cdot \pr^*g \in \smooth[0]_c\paren{V
    \setminus \lcS[(1)]}$, which is an extension of $g$ (note that
  $\pr^*g$ is independent of the variable $z_1$),
  and let $\sect$ be the canonical section of $S$ such that $\phi_S
  =\log\abs{\sect}^2$ and $\sect = z_1 \dotsm z_{\sigma_V}$ on $V$.
  Write $\sect =: z_1 \sect_{(1)}$ and $\phi_{(1)}
  :=\log\abs{\sect_{(1)}}^2$ for convenience.
  Let $(r_1, \theta_1)$ be the polar coordinates of the $z_1$-plane.
  Recall that, by the definition of admissible open sets (see
  \cite{Chan_adjoint-ideal-nas}*{\S 4.1}), one has
  $r_1^2 \fdiff{r_1^2}\psi > 0$ on $V$ and $\parres{r_1^2
    \fdiff{r_1^2}\psi}_{\set{r_1 = 0}} = \nu_1$.
  In view of Fubini's theorem, the norm with respect to the
  $1$-lc-measure can be computed as
  \begin{align*}
    \RTF[g]|f|(\eps)[X,1]
    &=\eps \int_X \frac{\rs g \:\abs{\rs*f}^2 \:e^{-\vphi_L-m_k\psi}}{\logpole//}
    =\eps \int_V \frac{\rs g \:\abs{\rs* f_0'}^2 \:e^{-\bphi
        -\phi_{(1)}}}{\abs{z_1}^2 \:\logpole//} \bigwedge_{j=1}^n
      \pi\ibar \:dz_j \wedge d\conj{z_{j}} \\
    &=\eps \int_W \int_U \frac{\rs g \:\abs{\rs* f_0'}^2 \:e^{-\bphi
      -\phi_{(1)}}}{r_1^2 \fdiff{r_1^2}\psi \:\logpole//}
      \:d \psi \: \frac{d\theta_1}{2} \cdot
      \bigwedge_{j=2}^n \pi\ibar \:dz_j \wedge d\conj{z_{j}} \\
    &=\int_W \int_U \frac{\rs g \:\abs{\rs* f_0'}^2 \:e^{-\bphi
      -\phi_{(1)}}}{r_1^2 \fdiff{r_1^2}\psi}
      \:d\paren{\frac{1}{\paren{\slog}^\eps}}  \frac{d\theta_1}{2} \cdot
      \dvol_{W} \\
    &\overset{\mathclap{\text{int.~by parts}}}= \quad
      -\int_W \int_U \fdiff{r_1} \paren{\frac{\rs g \:\abs{\rs* f_0'}^2 \:e^{-\bphi
      -\phi_{(1)}}}{r_1^2 \fdiff{r_1^2}\psi}} \:
      \frac{1}{\paren{\slog}^\eps}
      \:dr_1 \:\frac{d\theta_1}{2} \cdot
      \dvol_{W} \\
    &\hphantom{=}~\mathllap{\xrightarrow{\eps \tendsto 0^+}}~
      \frac{\pi}{\nu_1} \int_{V \cap D_1}
      g \:\abs{\rs* f_0'}^2 \:e^{-\bphi -\phi_{(1)}}
      \dvol_{V \cap D_1} \; .
  \end{align*}
  (Note that $e^{-\phi_{(1)}}$ is smooth on $\supp g$.)
  For the norm with respect to the Ohsawa measure, note that the norm
  is independent of the choice of the extension $\rs g$ of $g$ (see,
  for example, \cite{KimDano&Seo_adj-idl}*{Lemma 3.5} for a detailed
  proof).
  Choose $\rs g$ to be
  \begin{equation*}
    \rs g
    := \rho(z_1) \cdot \pr^*\paren{g \:\parres{\abs{\rs*f_0'}^2 e^{-\bphi}}_{D_1}}
    \frac{1}{\nu_1} \:
    \frac{
      r_1^2 \fdiff{r_1^2} \psi
    }{\abs{\rs*f_0'}^2 e^{-\bphi}} 
  \end{equation*}
  (one may take a further log-resolution such that $\divsr{\rs*f_0'} +
  S$ is in snc and the poles coming from $\frac{1}{\abs{\rs*f_0'}^2}$
  in the above formula can be cancelled out by
  $\pr^*\paren{\res{\abs{\rs*f_0'}^2}_{D_1}}$).
  Let $t$ be sufficiently negative such that $\rho \equiv 1$ on
  $\set{\psi < t+1} \cap \supp \rs g$ (recall the $g \in
  \smooth[0]_c\paren{D_1 \setminus \lcS[(1)]}$).
  It then follows from Fubini's theorem and integration by parts that
  \begin{align*}
    \smashoperator[r]{\int\limits_{\set{t < \psi < t+1}}} \;
    \rs g \:\abs{\rs* f}^2
    \:e^{-\vphi_L -m_k\psi} 
    &=\smashoperator[r]{\int\limits_{\set{t < \psi < t+1}}} \;
      \rs g \:\abs{\rs* f_0'}^2
      \:e^{-\bphi-\phi_{(1)}}
      \:\frac{\pi\ibar \:dz_1 \wedge d\conj{z_1}}{\abs{z_1}^2} \wedge
      \bigwedge_{j=2}^n \pi\ibar \:dz_j \wedge d\conj{z_{j}} \\
    &=\smashoperator[r]{\iint\limits_{\set{t < \psi < t+1}}} \;
      \frac{\rs g \:\abs{\rs* f_0'}^2 \:e^{-\bphi}}{r_1^2
      \fdiff{r_1^2} \psi}
      \:d\psi \: \frac{d\theta_1}{2} \cdot
      e^{ -\phi_{(1)}}
      \dvol_{W} \\
    &\overset{\mathclap{\text{int.~by parts}}}=
      \begin{aligned}[t]
        &~\quad\; \frac{1}{\nu_1}
        \int_W
        \res{\pr^*\paren{g \parres{\abs{\rs* f_0'}^2 \:e^{-\bphi}}_{D_1}}}_{\mathrlap{\set{\psi =t+1}}} \quad\;
        \paren{t+1} \:
        \frac{d\theta_1}{2} \cdot
        e^{ -\phi_{(1)}}
        \dvol_{W} \\
        &-\frac{1}{\nu_1}
        \int_W
        \res{\pr^*\paren{g \parres{\abs{\rs* f_0'}^2 \:e^{-\bphi}}_{D_1}}}_{\mathrlap{\set{\psi = t}}} \quad\; t \:
        \frac{d\theta_1}{2} \cdot
        e^{ -\phi_{(1)}}
        \dvol_{W} \\
        &-\frac{1}{\nu_1} \smashoperator{\iint\limits_{\set{t < \psi < t+1}}} \;
        \psi \:
        \cancelto{0}{
          \fdiff{r_1} \paren{\pr^*\paren{g \parres{\abs{\rs* f_0'}^2
                \:e^{-\bphi}}_{D_1}}}
        } \:dr_1 \:
        \frac{d\theta_1}{2} \cdot
        e^{ -\phi_{(1)}}
        \dvol_{W}
      \end{aligned}
    \\
    &=\frac{1}{\nu_1} \int_W \res{\pr^*\paren{g \parres{\abs{\rs*
      f_0'}^2 \:e^{-\bphi}}_{D_1}}}_{\set{\psi =t+1}}
      \frac{d\theta_1}{2} \cdot e^{ -\phi_{(1)}}
      \dvol_{W} 
    \\
    &\hphantom{=}~\mathllap{\xrightarrow{t \tendsto -\infty}}~
      \frac{\pi}{\nu_1} \int_{V \cap D_1}
      g \:\abs{\rs* f_0'}^2 \:e^{-\bphi -\phi_{(1)}}
      \dvol_{V \cap D_1} =\RTF[g]|f|(0)[X,1] \; .
  \end{align*}
  This completes the proof.
\end{proof}

\begin{cor} \label{cor:1-aidl=Demailly-intermediate-multidl}
  One has $\multidl'(m_{k-1}) = \aidlof|1|{\vphi_L}.{m_k}$ when
  $\vphi_L$ has only neat analytic singularities, where
  $\multidl'(m_{k-1})$ is the subsheaf of
  $\mtidlof{\vphi_L+m_{k-1}\psi}$ consisting of all the germs which
  are locally $L^2$ with respect to the Ohsawa measure, as defined in
  \cite{Demailly_extension}*{Def.~(2.11)}.
\end{cor}

\begin{proof}
  The results in \cite{Chan_adjoint-ideal-nas}*{Thm.~4.1.2} imply
  that, when $\vphi_L$ has only neat analytic singularities,
  \begin{equation*}
    \aidlof{\vphi_L}.{m_k}_x
    = \setd{f \in \mtidlof{\vphi_L +m_{k-1}\psi}_x}{
      \exists~\text{open } V \ni x \text{ such that }
      \RTF|f|(0)[V,\sigma] < +\infty
    }
  \end{equation*}
  for any integer $\sigma \geq 0$.
  The claim then follows directly from Proposition \ref{prop:1-lcV=Ohvol}.
\end{proof}

Set $\RTF|f|(\eps;\alert{\ell})[X,\sigma] := \eps \int_X
\frac{\abs f^2 \:e^{-\vphi_L-m_k\psi}}{\logpole.{\alert{\ell}}}$.
The above results indeed translate Theorem
\ref{thm:Demailly-extension-with-estimate} together with
\cite{Chan&Choi_ext-with-lcv-codim-1}*{Thm.~1.4.5} (see Remark
\ref{rem:extension-thm-with-1-lc-measure}) into a statement which says
that, under the given curvature assumption and the assumption that
$\vphi_L$ has only neat analytic singularities, there exists a
``universal constant'' $\ell \geq e$ such that every $f \in \cohgp
0[S]{K_X \otimes 
  L \otimes \frac{
    \aidlof|1|
  }{
    \aidlof|0|
  }}$ can be extended to a holomorphic section $F \in \cohgp 0[X]{K_X
  \otimes L \otimes \aidlof|1|}$ with an estimate
\begin{equation*}
  \RTF|F|(1;\ell)[X,1] \leq \RTF|f|(0)[X,1] \; .
\end{equation*}
Note that the residue norm $\RTF*|f|(0;\ell)[X,1]$ is independent of
$\ell \geq e$ (see
\cite{Chan_on-L2-ext-with-lc-measures}*{Cor.~2.3.3}) and thus the
argument $\ell$ is omitted.

In order to achieve the goal stated at the end of Section
\ref{sec:ext-thm-with-estimates} (after Remark
\ref{rem:extension-thm-with-1-lc-measure}), one is led to the search
of a proof of the following conjecture.

\begin{conjecture}[\cite{Chan_on-L2-ext-with-lc-measures}*{Conj.~1.1.3}]
  Under the curvature assumption given in Theorem
  \ref{thm:Demailly-extension-with-estimate}, there exists a
  sufficiently large constant $\ell \geq e$ (depending only on the
  function $\psi$ and the constant $\delta$ in the curvature
  assumption) such that, for any integer $\sigma \geq 1$, every
  \begin{equation*}
    f \in \cohgp 0[\lcc.{m_k}]{\; K_X \otimes L \otimes \frac{
        \aidlof.{m_k}
      }{
        \aidlof-1.{m_k}
      }}
  \end{equation*}
  has a holomorphic extension $F \in \cohgp 0[X]{K_X \otimes L \otimes
    \aidlof.{m_k}}$ satisfying the estimate
  \begin{equation*}
    \RTF|F|(1; \ell)[\lcdata*<X>.{m_k},\sigma]
    \leq \RTF|f|(0)[\lcdata*<X>.{m_k},\sigma] \; .
  \end{equation*}
\end{conjecture}

If the conjecture holds true, then, in view of the short exact
sequence 
\begin{equation*}
  \renewcommand{\objectstyle}{\displaystyle}
  \xymatrix@R=0.1cm{
    {0} \ar[r] &
    {\frac{
        \aidlof-1.{m_k}
      }{
        \aidlof|0|.{m_k}
      }} \ar[r] &
    {\frac{
        \aidlof.{m_k}
      }{
        \aidlof|0|.{m_k}
      }} \ar[r] &
    {\frac{
        \aidlof.{m_k}
      }{
        \aidlof-1.{m_k}
      }} \ar[r] &
    {0}
  }
\end{equation*}
for any integer $\sigma \geq 1$, the procedure similar to the one
described at the end of Section \ref{sec:ext-thm-with-estimates} will
result in extensions of any $f$ taking values in
$\frac{\aidlof|\sigma_{\mlc}|.{m_k}}{\aidlof|0|.{m_k}}
=\frac{\mtidlof{\vphi_L+m_{k-1}\psi}}{\mtidlof{\vphi_L+m_k\psi}}$ with
estimates in terms of the residue norms.
The estimates thus obtained are compatible with the estimate in the
example of a bidisc obtained by Berndtsson in
\cite{Cao&Paun_OT-ext}*{\S A.3}.

While the conjecture is still open,
\cite{Chan_adjoint-ideal-nas}*{Thm.~1.2.3(3) and Cor.~4.3.2} guarantee
a local version of the statement.
\begin{thm}[\cite{Chan_adjoint-ideal-nas}*{Thm.~1.2.3(3) and
    Cor.~4.3.2}] \label{thm:local-L2-extension}
  Suppose that $\vphi_L$ has only neat analytic singularities and
  assume that $\lcdata<X>$ is in an snc configuration (as in the proof
  of Proposition \ref{prop:1-lcV=Ohvol}).
  Let $\bphi$ be the potential given in
  \eqref{eq:potential-decomposition}.
  Then, on any admissible open set $V \subset X$ (see
  \cite{Chan_adjoint-ideal-nas}*{\S 4.1}), given a constant $C \geq
  0$ such that
  \begin{equation} \label{eq:bphi-boundedness-assumption}
    \res{\bphi}_{\lcS} \leq \bphi + C \quad\text{ on } V
  \end{equation}
  for every $\sigma$-lc centre $\lcS$ (where $\res{\bphi}_{\lcS}$ is
  the restriction of $\bphi$ to $\lcS$ and treated as a function on
  $V$), every $f \in \cohgp 0[\cl V]{K_X \otimes L \otimes
    \frac{\aidlof}{\aidlof-1}}$ has a holomorphic extension
  $F \in \cohgp 0[\cl V]{K_X \otimes L \otimes \aidlof}$ satisfying the
  estimate
  \begin{equation*}
    \RTF|F|(\eps)[V,\sigma] \;\;\paren{=\RTF|F|(\eps ; e)[V,\sigma]}\;\;
    \leq 2 e^C \:\RTF|f|(0)[V,\sigma]
    \quad\text{ for all } \eps \geq 0 \; .
  \end{equation*}
\end{thm}

The proof of Theorem \ref{thm:local-L2-extension} does not use the
machinery of the $L^2$ method as in the proof of the Ohsawa--Takegoshi
extension theorem.
Instead, it is obtained through a direct computation with the aid from
Taylor expansion (see \cite{Chan_adjoint-ideal-nas}*{Cor.~4.3.2} for
the proof).
Note that the constant $C \geq 0$ in
\eqref{eq:bphi-boundedness-assumption} in the theorem exists for any
quasi-psh $\bphi$ having only neat analytic singularities with snc
such that $\bphi^{-1}(-\infty)$ contains no $\sigma$-lc centres of
$\lcdata<V>$.

Although the constant in the estimate is not ``universal'' (as $C$
depends on $\bphi$, hence $\vphi_L$), it is worth noting that, if
$\bphi$ is \emph{psh and toric} on $V$, the mean-value-inequality
yields (suppose that $\bphi$ depends only on $(z_1,z_2)$ and consider
an $1$-lc centre $\lcS|1|$ given by $\set{z_1 =0}$ for example)
\begin{equation*}
  \res{\bphi}_{\lcS|1|} =\bphi(0,z_2) \leq \frac{1}{2\pi}
  \int_0^{2\pi} \bphi(r_1 e^{\cplxi \theta_1}, z_2)
  \:d\theta_1 =\bphi(z_1, z_2) \; ,
\end{equation*}
that means, the constant $C$ can be chosen to be $0$. 
Moreover, if the psh potential $\bphi$ has more general singularities
and does satisfy \eqref{eq:bphi-boundedness-assumption} for some $C
\geq 0$, it can be shown that each member of its Bergman kernel
approximation $\seq{\dep[k]\bphi}_{k \in\Nnum}$ also satisfies
\eqref{eq:bphi-boundedness-assumption} with the same constant $C \geq
0$.
This fact may be useful to compensate for the loss of the universal
constant in some applications.

Using a partition of unity, Theorem
\ref{thm:local-L2-extension} guarantees a \emph{smooth (global)
  extension with estimate} for every $f \in \cohgp 0[\lcc]{K_X \otimes
  L \otimes \frac{\aidlof}{\aidlof-1}}$.
It is hoped that this would be useful in constructing the desired
global holomorphic extension with estimate.



\end{document}